\documentclass{amsart}

\usepackage{graphicx,subfigure,amssymb}
\usepackage{color,xypic}

\newcommand{\ad}{\ensuremath{\mathrm{ad}}}

\newcommand{\Aut}{\ensuremath{\mathrm{Aut}}}
\newcommand{\NN}{\ensuremath{\mathbb{N}}}
\newcommand{\ZZ}{\ensuremath{\mathbb{Z}}}
\newcommand{\QQ}{\ensuremath{\mathbb{Q}}}
\newcommand{\CC}{\ensuremath{\mathbb{C}}}
\newcommand{\la}{\ensuremath{\langle}}
\newcommand{\ra}{\ensuremath{\rangle}}
\newcommand{\liea}[1]{\ensuremath{\mathfrak{#1}}}

\newcommand{\im}{\operatorname{im}}
\newcommand{\mL}{\mathcal{L}}
\newcommand{\mF}{\mathcal{F}}
\newcommand{\mI}{\mathcal{I}}
\newcommand{\gr}{\operatorname{gr}}
\newcommand{\gKM}{\liea{g}_{\text{KM}}}
\newtheorem{thm}{Theorem}
\newtheorem{lm}[thm]{Lemma}

\theoremstyle{definition}
\newtheorem{re}[thm]{Remark}
\theoremstyle{definition}
\newenvironment{ex}[1]{\vspace*{.2cm}\noindent\textbf{#1}.} {\vspace*{.2cm}}
\def\cprime{$'$}

\begin{document}
\title{Constructing simply laced Lie algebras from extremal elements}
\date{\today}
\author[J.~Draisma]{Jan Draisma}
\address[Jan Draisma]{
Department of Mathematics and Computer Science\\
Technische Universiteit Eindhoven\\
P.O. Box 513, 5600 MB Eindhoven, The Netherlands}
\thanks{The first author is supported by DIAMANT, an NWO
mathematics cluster.}
\thanks{
The second author is supported by NWO Ph.D. grant 10002490.}
\email{j.draisma@tue.nl}

\author[J.~C.~H.~W.~in~'t~panhuis]{Jos in 't panhuis}
\address[Jos in 't panhuis]{
Department of Mathematics and Computer Science\\
Technische Universiteit Eindhoven\\
P.O. Box 513, 5600 MB Eindhoven, The Netherlands}
\email{j.c.h.w.panhuis@tue.nl}
\maketitle

\begin{abstract}
For any finite graph $\Gamma$ and any field $K$ of characteristic
unequal to $2$ we construct an algebraic variety $X$ over $K$ whose
$K$-points parameterise $K$-Lie algebras generated by extremal elements,
corresponding to the vertices of the graph, with prescribed commutation
relations, corresponding to the non-edges. After that, we study the case
where $\Gamma$ is a connected, simply laced Dynkin diagram of finite
or affine type. We prove that $X$ is then an affine space, and that all
points in an open dense subset of $X$ parameterise Lie algebras isomorphic
to a single fixed Lie algebra. If $\Gamma$ is of affine type, then this
fixed Lie algebra is the split finite-dimensional simple Lie algebra
corresponding to the associated finite-type Dynkin diagram. This gives a
new construction of these Lie algebras, in which they come together with
interesting degenerations, corresponding to points outside the open dense
subset. Our results may prove useful for recognising these Lie algebras.
\end{abstract}

\section{Introduction and main results} \label{sec:Introduction}

An {\em extremal element} of a Lie algebra $\mathcal{L}$ over a
field $K$ of characteristic unequal to $2$ is an element $x \in
\mathcal{L}$ for which $[x,[x,\mathcal{L}]]\subseteq Kx$. A {\em
sandwich element} is an $x \in \mathcal{L}$ satisfying the stronger
condition $[x,[x,\mathcal{L}]]=0$.  The definition of extremal elements in
characteristic $2$ is more involved, which is one reason for restricting
ourselves to characteristics unequal to $2$ here. Extremal elements and
sandwich elements play important roles in both classical and modern Lie
algebra theory. In complex simple Lie algebras, or their split analogues
over other fields, extremal elements are precisely the elements that are
long-root vectors relative to some maximal torus. Sandwich elements are
used in the classification of simple Lie algebras in small characteristics
\cite{Premet97}; they occur in the modular Lie algebras of Cartan type,
such as the Witt algebras. Sandwich elements were originally introduced
in relation with the restricted Burnside problem \cite{Kostrikin81}. An
important insight for the resolution of this problem is the fact that a
Lie algebra generated by finitely many sandwich elements is necessarily
finite-dimensional. While this fact was first only proved under extra
assumptions, in \cite{Zelmanov90} it is proved in full generality. We
will use this result in what follows.

The prominence of extremal elements in the work of Kostrikin and
Zel{\cprime}manov and in modular Lie algebra theory led to the natural
problem of describing all Lie algebras generated by a fixed number of
extremal elements \cite{Cohen01,panhuis07,Postma07,Roozemond05}.

\begin{ex}{Example} 
Suppose
that we want to describe all Lie algebras $\mathcal{L}$ generated by two extremal
elements $x$ and $y$. Since $[x,[x,y]]$ is a scalar multiple $ax$
of $x$ and $[y,[x,y]]=-[y,[y,x]]$ is a scalar multiple $-by$ of $y$,
$\mL$ is spanned by $x,y,[x,y]$. There may be linear dependencies between
these elements, but let us assume that they are linearly independent. Then
\[ a[y,x]=[y,[x,[x,y]]]=[[y,x],[x,y]]+[x,[y,[x,y]]]=0-b[x,y]=b[y,x],
\]
and since we have assumed that $[x,y] \neq 0$, we find that $a=b$. Hence
three-dimensional Lie algebras with a distinguished pair of extremal
generators are parameterised by the single number $a$. Moreover, all
algebras with $a \neq 0$ are mutually isomorphic and isomorphic to the
split simple Lie algebra of type $A_1$, while the algebra with $a=0$ is
nilpotent and isomorphic to the three-dimensional Heisenberg algebra. This
is a prototypical example of our results. The next smallest case of
three generators is treated in \cite{Cohen01,Zelmanov90} and also by
our results below. There the generic Lie algebra is split of type $A_2$
and more interesting degenerations exist.
\end{ex}

We now generalise and formalise this example to the case of more
generators, where we also allow for the flexibility of prescribing
that certain generators commute. Thus let $\Gamma$ be a finite simple
graph without loops or multiple edges. Let $\Pi$ be the vertex set of
$\Gamma$ and denote the neighbour relation by $\sim$. Fixing a field $K$
of characteristic unequal to $2$, we denote by $\mF$ the quotient of the
free Lie algebra over $K$ generated by $\Pi$ modulo the relations
\[ [x,y]=0 \text{ for all } x,y \in \Pi \text{ with } x\not \sim y. \]
So $\mF$ depends both on $\Gamma$ and on $K$, but we will not make this
dependence explicit in the notation. We write $\mF^*$ for the space of all
$K$-linear functions $\mF \rightarrow K$. For every $f\in (\mF^*)^{\Pi}$,
also written $(f_x)_{x\in\Pi}$, we denote by $\mL(f)$ the quotient
of $\mF$ by the ideal $\mI(f)$ generated by the (infinitely many) elements
\begin{equation} \label{eq:Extremal}
[x,[x,y]]-f_x(y)x,\  x \in \Pi, y \in \mF.
\end{equation}
By construction $\mL(f)$ is a Lie algebra generated by extremal elements,
corresponding to the vertices of $\Gamma$, which commute when they are
not connected in $\Gamma$. The element $f_x$ is a parameter needed to
express the extremality of $x \in \Pi$. If $\Gamma$ is not connected,
then both $\mF$ and $\mL(f)$ naturally split into direct sums over
all connected components of $\Gamma$, so it is no restriction to assume
that $\Gamma$ is connected; we will do so throughout this paper.

In the Lie algebra $\mL(0)$ the elements of $\Pi$ map to
sandwich elements, hence by \cite{Zelmanov90} this Lie algebra is
finite-dimensional. For general $f\in (\mF^*)^{\Pi}$ it turns
out that $\dim \mL(f) \leq \dim \mL(0)$; see \cite{Cohen01} or
the proof of Theorem \ref{thm:Variety} below. It is therefore natural to
focus on the Lie algebras $\mL(f)$ of the maximal possible dimension
$\dim \mL(0)$.  This leads us to define the set
\[ X:=\{f\in (\mF^*)^{\Pi}
    \mid \dim \mL(f) = \dim \mL(0)\}, \]
the parameter space for all maximal-dimensional Lie algebras of the
form $\mL(f)$.

\begin{ex}{Example}
In the two-generator case above $\Gamma$ is the
graph with two vertices joined by an edge. The sandwich algebra $\mL(0)$
is the three-dimensional Heisenberg algebra, and the condition that
$\dim \mL(f)=3$ corresponds to our assumption above that $x,y,[x,y]$ be
linearly independent. This linear independence forced the parameters $a$
and $b$ to be equal. Here $X$ is the affine line with coordinate $a$. All
Lie algebras corresponding to points $a \neq 0$ are mutually isomorphic.
\end{ex}

Our first main result is that $X$ carries a natural structure of
an affine algebraic variety. To specify this structure we note that
$\mI(0)$ is a homogeneous ideal relative to the natural $\NN$-grading
that $\mF$ inherits from the free Lie algebra generated by $\Pi$.

\begin{thm} \label{thm:Variety}
The set $X$ is naturally the set of $K$-rational points of an
affine variety of finite type defined over $K$. This variety can be
described as follows. Fix any finite-dimensional homogeneous subspace $V$ of
$\mF$ such that $V + \mI(0)=\mF$. Then the restriction map
\[ X \rightarrow (V^*)^\Pi,\ f \mapsto (f_x|_V)_{x \in \Pi} \]
maps $X$ injectively onto the set of $K$-rational points of a closed
subvariety of $(V^*)^\Pi$. This yields a $K$-variety structure on $X$
which is independent of the choice of $V$.
\end{thm}

\begin{figure}
\subfigure[$A_n^{(1)}$]{\includegraphics[scale=.3]{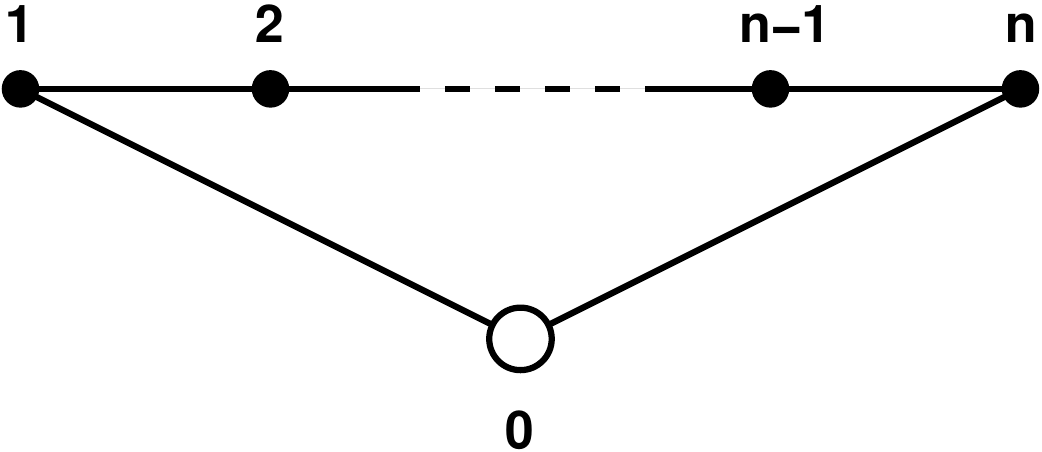}\label{fig:An}} \hfill
\subfigure[$D_n^{(1)}$]{\includegraphics[scale=.3]{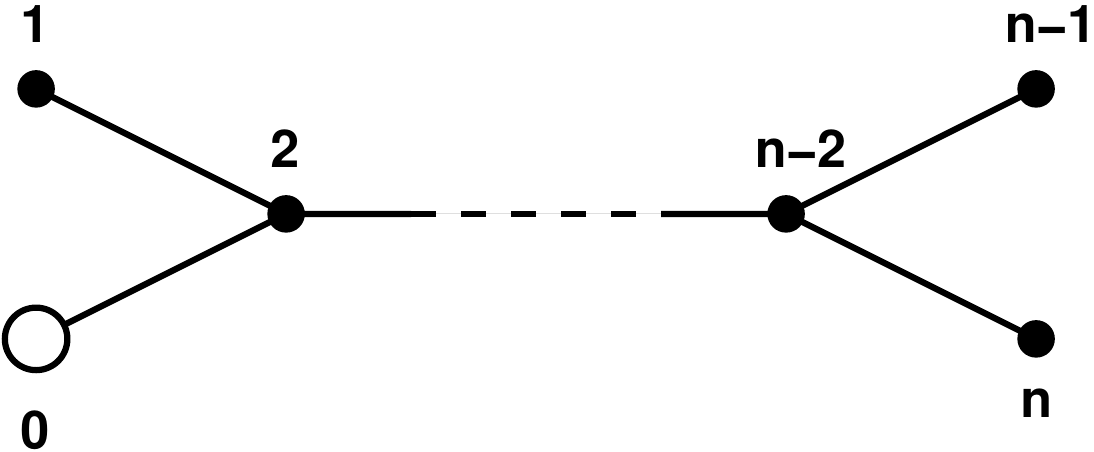}\label{fig:Dn}} \hfill
\subfigure[$E_6^{(1)}$]{\includegraphics[scale=.3]{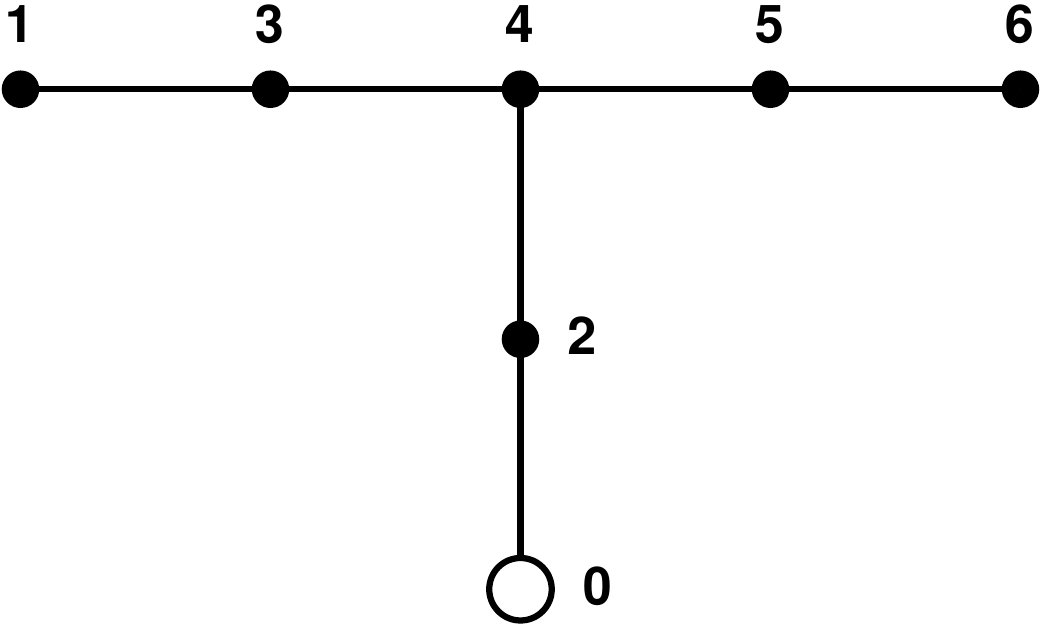}\label{fig:E6}} \hfill
\subfigure[$E_7^{(1)}$]{\includegraphics[scale=.3]{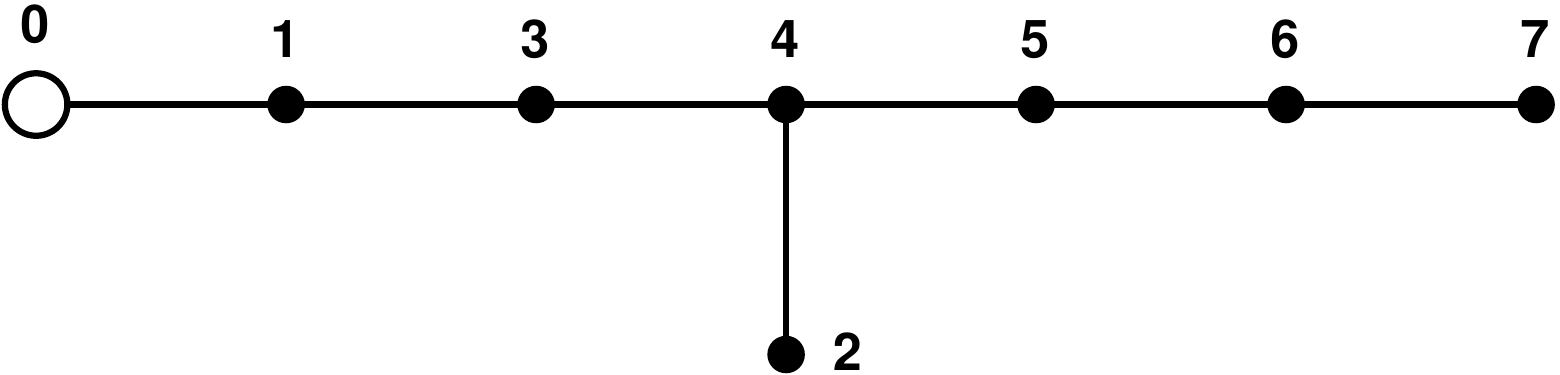}\label{fig:E7}} \hfill
\subfigure[$E_8^{(1)}$]{\includegraphics[scale=.3]{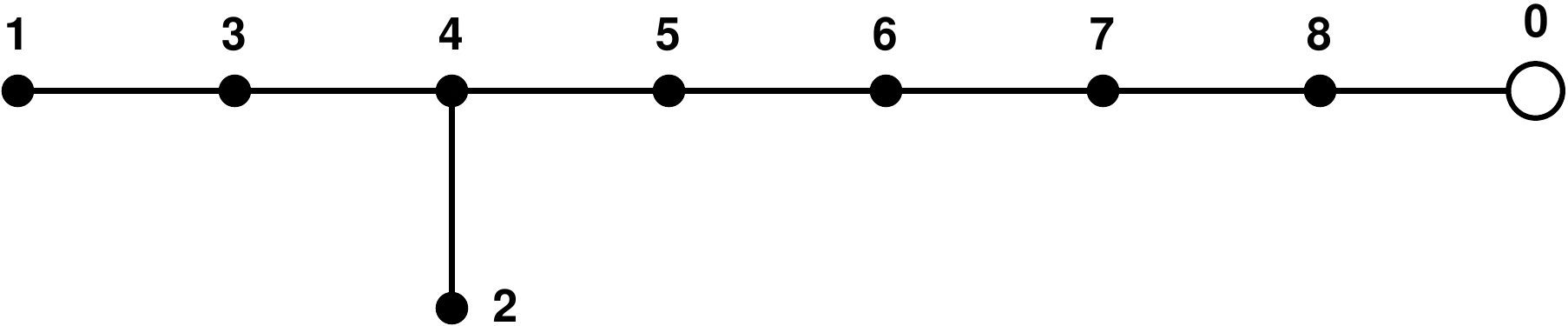}\label{fig:E8}}
\caption{The simply laced Dynkin diagrams of affine type. The notation comes from \cite{Kac90} and the corresponding
finite-type diagrams are obtained by deleting vertex $0$.}
\label{fig:Diagrams}
\end{figure}

We prove this theorem in Section \ref{sec:Moduli}. In Section
\ref{sec:Sandwich} we first derive some relations between the sandwich
algebra $\mL(0)$ and the positive part of the {\em complex} Kac-Moody
algebra of type $\Gamma$. Then we determine $\mL(0)$ explicitly in the
case where $\Gamma$ is a simply laced Dynkin diagram of finite or affine
type; by this we mean any of the diagrams in Figure \ref{fig:Diagrams}
without or with vertex $0$, respectively. See Theorems \ref{thm:Finite}
and \ref{thm:Affine}. In Section \ref{sec:Parameter} we study the variety
$X$. After some observations for general $\Gamma$, we again specialise to
the diagrams of Figure \ref{fig:Diagrams}. For these we prove that $X$
is an affine space, and that for $f$ in an open dense subset of $X$ the
Lie algebra $\mL(f)$ is isomorphic to a fixed Lie algebra; see Theorems
\ref{thm:GenericFinite} and \ref{thm:GenericAffine}. The latter of these
theorems is our second main result, and we paraphrase it here.

\begin{thm}
Let $\Gamma$ be any of the simply laced Dynkin diagrams of affine
type in Figure \ref{fig:Diagrams}, let $\Gamma^0$ be the finite-type
diagram obtained by removing vertex $0$ from $\Gamma$, and let $\Sigma$
be the edge set of $\Gamma$. Then $X$ is isomorphic to the affine space
of dimension $|\Sigma|+1$ over $K$, and for $f$ in an open dense subset
of $X$ the Lie algebra $\mL(f)$ is isomorphic to the Chevalley algebra
of type $\Gamma^0$.
\end{thm}

\begin{re} 
By the {\em Chevalley algebra of type $\Gamma^0$} we mean the Lie
algebra obtained by tensoring a certain $\ZZ$-form of the complex
simple Lie algebra of type $\Gamma^0$ with the field $K$; see Subsection
\ref{ssec:FiniteSandwich} for details. This Lie algebra is often simple,
but not always; see \cite[Chapter 4]{Strade04} and \cite{Seligman67}.
\end{re}

We conclude the paper with remarks on applications and related work
in Section \ref{sec:Notes}.

\section{The variety structure of the parameter space}
\label{sec:Moduli}

Recall the notation from Section \ref{sec:Introduction}: $\Gamma$
is a connected finite graph without loops or multiple edges, $K$
is a field of characteristic unequal to $2$, and $X$ is the set of
all $f\in(\mF^*)^{\Pi}$ such that $\mL(f)$ has the maximal possible
dimension, namely that of $\mL(0)$.  To avoid formulas with many
Lie brackets, we write $x_d \cdots x_1$ for the expression
$[x_d,[\ldots[x_2,x_1]\ldots]]$. Such an element is called a {\em
monomial} in the $x_i$ of {\em degree} $d$.  In the proof of Theorem
\ref{thm:Variety} we use the $\NN$-grading $\mF=\bigoplus_{d=1}^\infty
\mF_d$ of $\mF$ where $\mF_d$ is the span of all monomials of degree $d$
in the elements of $\Pi$. We also use the following terminology:
A subspace $V$ of $\mF$ is called {\em homogeneous} if it equals
$\bigoplus_d (V \cap \mF_d)$. A single element of $\mF$ is {\em
homogeneous} if it lies in some $\mF_d$. To any subspace $V$ of $\mF$
we associate the homogeneous subspace $\gr V$ of $\mF$ spanned by all
$v_d$ as $v=v_1+\ldots+v_d,\ v_i \in \mF_i,$ runs through $V$. If $V$
is an ideal then so is $\gr V$.

\begin{proof}[Proof of Theorem \ref{thm:Variety}]
Let $V$ be a finite-dimensional homogeneous subspace of $\mF$ such
that $\mF=V \oplus \mI(0)$; such a subspace exists as $\mL(0)$
is finite-dimensional \cite{Cohen01,Zelmanov90} and $\mI(0)$ is
homogeneous. Note that the theorem only requires that $\mF=V+\mI(0)$;
we will argue later why this suffices.
Observe that $V$ contains the image of $\Pi$ in $\mL$: the
abelian Lie algebra spanned by $\Pi$ is clearly a quotient of $\mL(0)$,
so the component of $\mI(0)$ in degree $1$ is trivial.  From the shape of
the generators \eqref{eq:Extremal} it is clear that the homogeneous ideal
$\gr \mI(f)$ associated to $\mI(f)$ contains $\mI(0)$, so that $\mF=V +
\mI(f)$ for {\em all} $f$, and $\mF=V \oplus \mI(f)$ if and only if $f
\in X$. We will argue that the map
\[ \Psi: X \rightarrow (V^*)^\Pi,\ f \mapsto (f_x|_V)_{x\in\Pi}=:f|_V \]
is injective, and that its image is a closed subvariety of
$(V^*)^\Pi$.

For each $f \in X$ let $\pi_f: \mF \rightarrow V$ be the projection onto $V$
along $\mI(f) $. We prove two slightly technical statements: First, for all
$u \in \mF$ there exists a polynomial map $P_u:(V^*)^\Pi \rightarrow V$
such that
\[ P_u(f|_V)=\pi_f(u) \text{ for all } f \in X; \]
and second, for $x \in \Pi$ and $u \in \mF$ there exists a polynomial
$Q_{x,u}:(V^*)^\Pi \rightarrow K$ such that
\[ Q_{x,u}(f|_V)=f_x(u) \text{ for all } f \in X \text{ and } Q_{x,u}(h)=h_x(u)
\text{ if } u \in V \text{ and } h \in (V^*)^\Pi. \]
We proceed by induction on the degree of $u$: assume that both statements
are true in all degrees less than $d$, and write $u=u_1+u_2+u_3$ where
$u_1$ has degree less than $d$, $u_2 \in V \cap \mF_d$, and $u_3 \in
\mI(0) \cap \mF_d$. Then $u_3$ can be written as a sum of terms of
the form $x_k \cdots x_1 x_1 u'$ with $x_i \in \Pi$ and $u'$ of degree
$d-(k+1)<d$. Modulo $\mI(f)$ for $f \in X$ this term is equal to
\[ f_{x_1}(u') \pi_f(x_k\cdots x_1)=Q_{x_1,u'}(f|_V) P_{x_k \cdots x_1}(f|_V), \]
where we used the induction hypothesis for $u'$ and $x_k\cdots x_1$.
Hence a $P_u$ of the form
\[ P_u:=P_{u_1}+u_2+\text{ terms of the form } Q_{x_1,u'} P_{x_k \cdots x_1} \]
has the required property. Similarly, for $x \in \Pi$ and $f \in X$
we have
\[ f_x(x_k\cdots x_1 x_1 u')x=xxx_k\cdots x_1 x_1 u'=
	Q_{x_1,u'}(f|_V) Q_{x,x_k\cdots x_1}(f|_V) x 
\mod \mI(f),\] 
and since $x \not \in \mI(f)$ we conclude that 
\[ f_x(x_k\cdots x_1 x_1 u')=Q_{x_1,u'}(f|_V) Q_{x,x_k\cdots x_1}(f|_V). \]
Hence we may define $Q_{x,u}$ by
\[ Q_{x,u}(h):=Q_{x,u_1}(h)+h_x(u_2)+\text{ terms of the form }
Q_{x_1,u'}(h) Q_{x,x_k \cdots x_1}(h),\ h \in (V^*)^\Pi. \]
This shows the existence of $P_u$ and $Q_{x,u}$. The injectivity of $\Psi$
is now immediate: any $f \in X$ is determined by its restriction to $V$
by $f_x(u)=Q_{x,u}(f|_V)$.

We now show that $\im(\Psi)$ is closed. For any tuple $h \in (V^*)^\Pi$
one may try to define a Lie algebra structure on $V$ by setting
\[ [u,v]_h:=P_{[u,v]}(h), \quad u,v \in V.  \]
By construction, if $h=f|_V$ for some $f \in X$, then this turns $V$
into a Lie algebra isomorphic to $\mL(f)$. Moreover, in this case the
Lie bracket has the following two properties:
\begin{enumerate}
\item If $v \in V$ is expressed as a linear combination
$\sum_{x_1,\ldots,x_d\in\Pi}c_{(x_d,\ldots,x_1)} x_d\cdots x_1$ of
monomials in the elements of $\Pi$, where the Lie bracket is taken in
$\mF$, then the expression
$\sum_{x_1,\ldots,x_d\in\Pi}c_{(x_d,\ldots,x_1)}[x_d,[\ldots [x_2,x_1]_h\ldots]_h]_h$ also equals $v$; and
\label{it:gen}
\item $[x,[x,u]_h]_h=Q_{x,u}(h) x$ for all $x\in \Pi,u \in
V$. \label{it:ext}
\end{enumerate}
Conversely, suppose that $[.,.]_h$ indeed defines a Lie algebra on $V$
satisfying \eqref{it:gen} and \eqref{it:ext}. Then $(V,[.,.]_h)$ is a Lie
algebra of dimension $\dim \mL(0)$ that by \eqref{it:gen} is generated
by the image of $\Pi$, and by \eqref{it:ext} this image consists of
extremal elements. Hence there exists an $f \in X$ corresponding to this
Lie algebra, and its restriction to $V$ is $h$---indeed, $f_x(u)$ is
the coefficient of $x$ in $[x[x,u]_h]_h$, which is $Q_{x,u}(h)=h_x(u)$
for $u \in V$. Finally, all stated conditions on $h$---the fact that
$[.,.]_h$ satisfies the Jacobi identity and anti-commutativity, together
with \eqref{it:gen} and \eqref{it:ext}---are closed; here we use the
polynomiality of $P_u$ and $Q_{x,u}$. This proves that $\im(\Psi)$
is closed.

Now if $U$ is any homogeneous subspace containing $V$, then the restriction
map $\Psi':X \rightarrow (U^*)^\Pi$ is clearly also injective. Moreover,
an $h' \in (U^*)^\Pi$ lies in the image of this map if and only if $h'|_V$
lies in $\im \Psi$ and $h'_x(u)=Q_{x,u}(h'|_V)$ for all $u \in U$. Thus
$\im \Psi'$ is closed and the maps $\im \Psi' \rightarrow \im \Psi,\
h' \mapsto h'|_V$ and $\im \Psi \rightarrow \im \Psi',\ h \mapsto h'$
with $h'_x(u)=Q_{x,u}(h),\ u \in U$ are inverse morphisms between $\im
\Psi$ and $\im \Psi'$. Similarly, if $V'$ is any other homogeneous
vector space complement of $\mI(0)$ contained in $U$, then the restriction map 
$(U^*)^\Pi \rightarrow ((V')^*)^\Pi$ induces an isomorphism between the
images of $X$ in these spaces. This shows that the variety structure of
$X$ does not depend on the choice of $V$. Finally, all morphisms indicated
here are defined over $K$. We conclude that we have a $K$-variety structure on $X$
which is independent of the choice of $V$.
\end{proof}

The type of reasoning in this proof will return in Section
\ref{sec:Parameter}: in the case where $\Gamma$ is a Dynkin diagram we
will show that for $f \in X$ the restriction $f|_V$ actually depends
polynomially on even fewer values of the $f_x$, thus embedding $X$ into
smaller affine spaces. That these embeddings are closed can be proved
exactly as we did above.

\section{The sandwich algebra} \label{sec:Sandwich}

For now, $\Gamma$ is an arbitrary finite graph (not necessarily a Dynkin
diagram). The Lie algebra $\mathcal{L}(0)$ is the so-called {\em sandwich
algebra} corresponding to $\Gamma$. It is a finite-dimensional nilpotent
Lie algebra, and carries an $\NN^\Pi$-grading defined as follows. The
{\em weight} of a word $(x_d,\ldots,x_1)$ over $\Pi$ is the element $\mu
\in \NN^\Pi$ whose $x$-coordinate equals
\[ | \{i \in \{1,\ldots,d\} \mid x_i=x\}|, \]
for all $x \in \Pi$. For such a word the corresponding monomial $x_d\ldots
x_1$ lives in the free Lie algebra on $\Pi$, but we use the same notation
for its images in $\mathcal{F}$ and $\mathcal{L}(f)$ when this does
not lead to any confusion.  We will sometimes say that a monomial $x_d
\cdots x_1 \in \mathcal{L}(0)$ has weight $\mu$; then we mean that the
word $(x_d,\ldots,x_1)$ has weight $\mu$---the monomial $x_d \cdots x_1$
itself might be $0$.  Now the free Lie algebra is graded by weight,
and this grading refines the grading by degree.  Like the grading
by degree, the grading by weight is inherited by $\mathcal{L}(0)$
as all relations defining $\mathcal{L}(0)$ are monomials. We write
$\mathcal{L}(0)_{\mu}$ for the space of weight $\mu\in\NN^{\Pi}$ and
call $\dim \mathcal{L}(0)_{\mu}$ the {\em multiplicity} of $\mu$.

For $x \in \Pi$ let $\alpha_x$ be the element with a $1$ on position $x$
and zeroes elsewhere; that is, $\alpha_x$ is the weight of the word $(x)$.
We define a symmetric $\ZZ$-bilinear form $\la.,.\ra$ on
$\ZZ^\Pi$
by its values on the standard basis: for $x,y \in \Pi$ we set
\[ \la \alpha_x, \alpha_y \ra:=\begin{cases}
    2& \text{ if } x=y\\
    -1&\text{ if } x \sim y, \text{ and}\\
    0& \text{ otherwise.}
    \end{cases}
\]
The matrix $A:=(\la \alpha_x, \alpha_y \ra )_{x,y \in \Pi}$ is called
the {\em Cartan matrix} of $\Gamma$. The {\em height} of an element of
$\ZZ^\Pi$ is by definition the sum of the coefficients of the $\alpha_x,
x \in \Pi$, in it.

In what follows we often need to show that certain  monomials $x_d
\cdots x_1$ are zero in $\mathcal{L}(0)$. Lemmas \ref{lm:AtLeastTwo}--\ref{lm:VeryReal}
show how the bilinear form comes into play. But first we recall an
elementary property of sandwich elements, to which they owe their name.

\begin{lm} \label{lm:Sandwich}
Let $x$ be a sandwich element in a Lie algebra $\mathcal{L}$ and let $y,z
\in \mathcal{L}$ be arbitrary. Then $xyxz=0$.
\end{lm}

\begin{proof}
We have
\begin{align*}
xyxz&=[x,y]xz+yxxz=[x,y]xz=-[x,xy]z+x[x,y]z=zxxy+x[x,y]z\\
&=x[x,y]z=xxyz-xyxz=-xyxz.\qedhere
\end{align*}
\end{proof}

\begin{re}
Note that we have used here that the characteristic is not $2$. In the
case of characteristic $2$, this lemma should be taken as part of the
definition of a sandwich element \cite{Cohen06}.
\end{re}

\begin{lm} \label{lm:AtLeastTwo}
Let $w=(x_d,x_{d-1},\ldots,x_1)$ be a word over $\Pi$ and let $x \in
\Pi$. Let $x_i$, $x_j$ be consecutive occurrences of $x$ in $w$ (i.e.,
$i>j$, $x_i=x_j=x$, and $x_k \neq x$ for all $k$ strictly between $i$ and
$j$). Suppose that the letters in $w$ strictly between $x_i$ and $x_j$
contain at most $1$ occurrence of a $\Gamma$-neighbour of $x$, i.e.,
the set $\{k \in \{j+1,\ldots,i-1\} \mid x_k \sim x\}$ has cardinality
at most $1$. Then $x_d x_{d-1}\cdots x_1$ is $0$ in $\mathcal{L}(0)$.
\end{lm}

\begin{proof}
Set $z:=x_d x_{d-1}\cdots x_1$. First, using the fact that on $\mathcal{F}$
the linear map $\ad(x)$ commutes with $\ad(y)$ for any $y \in \Pi$ with $x
\not\sim y$, we can move $x_i$ in $z$ to the right until it is directly
to the left of either $x_j$ or the unique $x_k \sim x$ between $x_i$
and $x_j$, so we may assume that this was already the case to begin with.

If $i=j+1$ then either $j=1$ and $z$ is zero by anti-commutativity, or
$j>1$ and the monomial $x_i x_j x_{j-1} \cdots x_1=x x x_{j-1} \cdots x_1$
is zero by the sandwich property of $x$.

Suppose, on the other hand, that $x_i \sim x_{i-1}$. Then if $j=1$ and
$i-1>2$ the monomial $z$ is zero since $x_2 x_1$ is---indeed, $x_2 \not
\sim x_1$. On the other hand, if $i-1=2$ or $j>1$ then we can move $x_j$
in $z$ to the left until it is directly to the right of $x_{i-1}$. So
again, we may assume that it was there right from the beginning. But now
\[ x_i x_{i-1} x_j x_{j-1} \cdots x_1 = x x_{i-1} x x_{j-1} \cdots
x_1. \]
If $j>1$, then this monomial equals zero by Lemma \ref{lm:Sandwich};
and if $j=1$, then it is zero by the sandwich property of $x$.
\end{proof}

\begin{lm} \label{lm:Reduce}
Let $(x_d,x_{d-1},\ldots,x_1)$ be a word with $d \geq 2$ over $\Pi$
and suppose that the weight $\mu$ of $(x_{d-1},x_{d-2},\ldots,x_1)$ satisfies
$\la \alpha_{x_d}, \mu \ra \geq 0$. Then $x_d x_{d-1}\cdots x_1=0$ in $\mathcal{L}(0)$.
\end{lm}

\begin{proof}
Set $w:=(x_d,x_{d-1},\ldots,x_1)$ and $z:=x_dx_{d-1} \cdots x_1$. The condition
on the bilinear form can be written as follows:
\[ 2 |\{j \in \{1,\ldots,d-1\} \mid x_j=x_d\}|
    \geq |\{j \in \{1,\ldots,d-1\} \mid x_j \sim x_d\}|
\]
First we note that if the right-hand side is $0$, then $z$ is trivially
zero: then all $x_i$ with $i<d$ commute with $x_d$, and there at least
$d-1\geq 1$ such factors. So we may assume that the right-hand side is
positive, and hence so is the left-hand side.

Let the set in the left-hand side of this inequality consist
of the indices $i_m>i_{m-1}>\ldots>i_1$; by the above $m$ is
positive. In the word $w$ there are $m$ pairs $(i,j)$
satisfying the conditions of Lemma \ref{lm:AtLeastTwo} with $x=x_d$, namely,
$(d,i_m),$ $(i_m,i_{m-1}),$ $\ldots,$ $(i_2,i_1)$. Now if for some such
$(i,j)$ there are less than two $\Gamma$-neighbours of $x_d$ in
the interval between $x_i$ and $x_j$, then $z=0$ by Lemma
\ref{lm:AtLeastTwo}. So we may assume that each of these $m$ intervals
contains at least two
$\Gamma$-neighbours
of $x_d$. But then, by the above
inequality, these exhaust all $\Gamma$-neighbours of $x_d$ in $w$, so in
particular there are exactly $2$
$\Gamma$-neighbours
of $x_d$ between $x_{i_2}$
and $x_{i_1}$, and none to the right of $x_{i_1}$. Now if $i_1>1$, then
$z$ is zero because $x_{i_1}$ commutes with everything to the right of
it. Hence assume that $i_1=1$, and note that $i_2 \geq 4$.  If $x_2 \not
\sim x_1=x_{i_1}$, then again $z$ is trivially $0$, so assume that $x_2$
is a $\Gamma$-neighbour of $x_d=x_1$. Then we have
\[ x_{i_2} \cdots x_3 x_2 x_1 = - x_{i_2} \cdots x_3 x_1 x_2; \]
but in the monomial on the right there is only one $\Gamma$-neighbour
of $x_d$ between $x_{i_2}$ and $x_1$---hence it is zero by Lemma
\ref{lm:AtLeastTwo}.
\end{proof}

\begin{lm} \label{lm:MultOne}
Let $x \in \Pi$ and $\lambda \in \NN^\Pi$ satisfy $\la
\alpha_x, \lambda \ra=-1$. Then $\mathcal{L}(0)_{\alpha_x+\lambda}=[x,\mathcal{L}(0)_\lambda]$.
\end{lm}

\begin{proof}
Let $w=(x_d,\ldots,x_1)$ be a word over $\Pi$ of weight
$\alpha_x+\lambda$.  We show that in $\mathcal{L}(0)$ the monomial $z:=x_d \cdots x_1$
is a scalar multiple of some monomial of the form $x z'$, where $z'$
is a monomial of weight $\lambda$. Obviously, $x$ occurs in $w$; let $k$
be maximal with $x_k=x$. If $k=1$, then we may interchange $x_k=x_1$ and
$x_{k+1}=x_2$ in $z$ at the cost of a minus sign (note that $d \geq 2$),
so we may assume that $k \geq 2$.

Suppose first that there occur $\Gamma$-neighbours of $x=x_k$ to the
left of $x_k$ in $w$. We claim that then $z=0$. Indeed, let $\mu,\nu$
be the weights of $(x_{k-1},\ldots,x_1)$ and $(x_d,\ldots,x_{k+1})$,
respectively. Then we have
\[ \la \alpha_x,\mu \ra = \la \alpha_x,\lambda \ra
- \la \alpha_x, \nu \ra = -1 - \la \alpha_x,\nu \ra \geq 0, \]
where in the last inequality we use that there are occurrences
of neighbours of $x_k$, but none of $x_k$ itself, in the word
$(x_d,\ldots,x_{k+1})$. Now we find $x_k \cdots x_1=0$ by Lemma
\ref{lm:Reduce} (note that $k\geq 2$), hence $z=0$ as claimed.

So we can assume that there are no $\Gamma$-neighbours of $x_k$ to the
left of $x_k$ in $w$.  Then we may move $x_k$ in $z$ all the way to
the left, hence $z$ is indeed equal to $xz'$ for some monomial $z'$
of weight $\lambda$.
\end{proof}

\subsection{Relation with the root system of the Kac-Moody algebra}
\label{ssec:KacMoody}

Recall the definition of the Kac-Moody algebra $\gKM$ over $\CC$
corresponding to $\Gamma$: it is the Lie algebra generated by $3 \cdot
|\Pi|$ generators, denoted $E_x,H_x,F_x$ for $x \in \Pi$, modulo the
relations
\begin{align*}
&H_x H_y = 0,\quad E_x F_x=H_x,\\
&H_x E_y =\la \alpha_x,\alpha_y \ra E_y,\quad H_x F_y=-\la
\alpha_x,\alpha_y \ra F_y; \text{ and for } x \neq y:\\
&E_x F_y=0,\  \ad(E_x)^{1-\la \alpha_x,\alpha_y \ra}E_y=0,
\text{ and } 
\ad(F_x)^{1-\la \alpha_x,\alpha_y \ra}F_y=0.
\end{align*}
Endow $\gKM$ with the $\ZZ^\Pi$-grading in
which $E_x,H_x,F_x$ have weights $\alpha_x,0,-\alpha_x$, respectively.
Let $\Phi:=\{\beta \in \ZZ^\Pi \setminus\{0\} \mid (\gKM)_\beta \neq
0\}$ be the {\em root system} of $\gKM$; it is equal to the disjoint
union of its subsets $\Phi_\pm:=\Phi \cap (\pm \NN)^\Pi$ and contains
the {\em simple roots} $\alpha_x,\ x \in \Pi$; we refer to \cite{Kac90}
for the theory of Kac-Moody algebras.  In what follows we will compare
the multiplicities of weights in the $K$-algebra $\mL(0)$ and the
$\CC$-algebra $\gKM$.

\begin{lm} \label{lm:Roots}
For $\lambda \in \NN^\Pi \setminus \Phi_+$ we have $\mL(0)_\lambda=0$.
\end{lm}

\begin{proof}
We proceed by induction on the height of $\lambda$. The proposition is
trivially true for $\lambda$ of height $1$. Suppose now that it is true
for height $d-1\geq 1$, and consider a word $w=(x_d,x_{d-1},\ldots,x_1)$
of weight $\lambda \not \in \Phi_+$.

Set $\mu:=\lambda-\alpha_{x_d}$.  If $\mu \not \in \Phi_+$, then
$x_{d-1}\cdots x_1=0$ by the induction hypothesis, so we may assume
that $\mu \in \Phi_+$. This together with $\mu+\alpha_{x_d} \not \in
\Phi_+$ implies (by elementary $\liea{sl}_2$-theory in $\gKM$) that $\la
\alpha_{x_d}, \mu \ra \geq 0$. Now Lemma \ref{lm:Reduce} shows that $x_d
\cdots x_1=0$.
\end{proof}

For another relation between weight multiplicities in $\mL(0)$ and $\gKM$,
recall that a root in $\Phi$ is called {\em real} if it is in the orbit
of some simple root under the Weyl group $W$ of $\gKM$. In that case it
has multiplicity $1$ in $\gKM$. We now call a root $\beta \in \Phi_+$
{\em very real}---this is non-standard terminology---if it can be written
as $\beta=\alpha_{x_d}+\ldots+\alpha_{x_1}$, for some $x_1,\ldots,x_d
\in \Pi$, such that for all $i=2,\ldots,d$ we have
\[ \la \alpha_{x_i}, \alpha_{x_{i-1}}+\ldots+\alpha_{x_1} \ra = -1. \]
(This implies that $\beta=s_{x_d} \cdots s_{x_2} \alpha_{x_1}$,
where the $s_{x}$ are the fundamental reflections corresponding
to the $x \in \Pi$.)

\begin{lm} \label{lm:VeryReal}
Any very real $\beta \in \Phi_+$ has multiplicity at most
$1$ in $\mL(0)$.
\end{lm}

\begin{proof}
This follows by induction on the height of $\beta$, using Lemma
\ref{lm:MultOne} for the induction step.
\end{proof}

\subsection{Simply laced Dynkin diagram of finite type}
\label{ssec:FiniteSandwich}

In this section we assume that $\Gamma$ is a Dynkin diagram of
finite type, i.e., one of the diagrams in Figure \ref{fig:Diagrams}
with vertex $0$ removed. Then $\gKM$ is a finite-dimensional simple
Lie algebra over $\CC$. Now $\gKM$ has a {\em Chevalley basis}
\cite[Section 4.2]{Carter72}. This basis consists of the images of
the $H_x$ and one vector $E_\alpha \in (\gKM)_\alpha$ for every root
$\alpha \in \Phi$, where $E_{\alpha_x},E_{-\alpha_x}$ may be taken
as $E_x,F_x$, respectively. An important property that we will need
is that $[E_{\alpha},E_{\beta}] = \pm E_{\alpha+\beta}$ for all roots
$\alpha,\beta$ such that $\alpha+\beta$ is a root; here we use that the
$p$ in \cite[Theorem 4.2.1]{Carter72} is $0$ in the simply laced case.
The Chevalley basis spans a $\ZZ$-subalgebra of $\gKM$. Let $\liea{g}$
be the $K$-algebra obtained by tensoring this $\ZZ$-form with $K$, and
let $E_x^0,H_x^0,F_x^0$ be the images in $\liea{g}$ of $E_x,H_x,F_x$.
The Lie algebra $\liea{g}$ has a triangular decomposition
\[ \liea{g}=\liea{n}_- \oplus \liea{h} \oplus \liea{n}_+, \]
where $\liea{n}_{\pm}:=\bigoplus_{\beta \in \Phi_\pm} \liea{g}_\beta$.
We will refer to $\liea{g}$ as the {\em Chevalley algebra of type
$\Gamma$}.

\begin{re} 
One can also define $\liea{g}$ as the Lie algebra of the split simply
connected algebraic group over $K$ of type $\Gamma$.
\end{re}

\begin{thm} \label{thm:Finite}
Let $\Gamma$ be a simply laced Dynkin diagram of finite type, obtained
from a diagram in Figure \ref{fig:Diagrams} by removing vertex
$0$. Let $\liea{g}$ be the corresponding Chevalley algebra over the
field $K$ of characteristic unequal to $2$, and let $\liea{n}_+$ be
the subalgebra generated by the $E_x^0$. Then the map sending $x \in
\Pi$ to $E_x^0$ induces a (necessarily unique) isomorphism $\mL(0)
\rightarrow \liea{n}_+$.
\end{thm}

In the proof of this theorem we use the following well-known facts about
simply laced Kac-Moody algebras of finite type: first, $\la .,. \ra$
only takes the values $-1,0,1,2$ on $\Phi_+ \times \Phi_+$, and second,
all roots in $\Phi_+$ are very real.

\begin{proof}[Proof of Theorem \ref{thm:Finite}.]
To prove the existence of a {\em homomorphism} $\pi$ sending $x$
to $E_x^0$, we verify that the relations defining $\mL(0)$ hold in
$\liea{n}_+$. That is, we have to prove that
\begin{align*}
&[E_x^0,E_y^0]=0\textrm{ for all }x,y\in\Pi \text{ with } x \not\sim y,\\
&\ad(E_x^0)^2 z=0\textrm{ for all }x\in\Pi\textrm{ and all }z\in\liea{n}_+.
\end{align*}
The first statement is immediate from the relations defining $\gKM$. For
the second relation, if $z \in \liea{n}_+$ is a root vector with root
$\beta \in \Phi_+$, then $\la \beta,\alpha_x \ra \geq -1$ by the above, so that
$\la \beta+2\alpha_x,\alpha_x \ra \geq 3$ and therefore $\beta+2\alpha_x \not
\in \Phi_+$, so that $\ad(E_x)^2 z=0$. As root vectors span $\liea{n}_+$,
we have proved the existence of $\pi$; uniqueness is obvious.

Now we have to show that $\pi$ is an isomorphism.  It is surjective
as $\liea{n}_+$ is generated by the $E_x^0$; this follows
from the properties of the Chevalley basis in Subsection
\ref{ssec:FiniteSandwich}. Hence it suffices to prove that $\dim
\mL(0) \leq \dim \liea{n}_+$. But by Lemma \ref{lm:Roots}, Lemma
\ref{lm:VeryReal} and the fact that all roots are very real we have
$\mL(0)_\mu=0$ for all $\mu \not \in \Phi_+$ and $\dim\mL(0)_\beta \leq
\dim \liea{g}_\beta$ for all $\beta \in \Phi_+$. This concludes the proof.
\end{proof}

\subsection{Simply laced Dynkin diagrams of affine type}
\label{ssec:AffineSandwich}

Suppose now that $\Gamma$ is a simply laced Dynkin diagram of affine
type from Figure \ref{fig:Diagrams}. Recall that the Cartan matrix $A$
has a one-dimensional kernel, spanned by a unique primitive vector
$\delta \in \NN^\Pi$. Here primitive means that the greatest common
divisor of the coefficients of $\delta$ on the standard basis is $1$;
indeed, there always exists a vertex $x_0 \in \Pi$ (labelled $0$ in
Figure \ref{fig:Diagrams}) with coefficient $1$ in $\delta$, and all
such vertices form an $\Aut(\Gamma)$-orbit. For later use, we let $h$
be the {\em Coxeter number}, which is the height of $\delta$.

Write $\Pi^0:=\Pi \setminus \{x_0\}$, $\Gamma^0$ for the induced
subgraph on $\Pi^0$ (which is a Dynkin diagram of finite type), and
$\Phi^0$ for the root system of the Chevalley algebra $\liea{g}$ of
type $\Gamma^0$ defined in Subsection \ref{ssec:FiniteSandwich}.
This root system lives in the space $\ZZ^{\Pi^0}$, which we
identify with the elements of $\ZZ^{\Pi}$ that are zero on
$x_0$. Retain the notation $\liea{n}_{\pm} \subseteq \liea{g}$
from Subsection \ref{ssec:FiniteSandwich}. Consider the semi-direct
product $\liea{u}:=\liea{n}_+ \ltimes \liea{g}/\liea{n}_+$, where the
second summand is endowed with the trivial Lie bracket and the natural
$\liea{n}_+$-module structure. This $\liea{u}$ is clearly a nilpotent Lie
algebra; we will prove that it is isomorphic to $\mL(0)$. In our proof
we use the following $\ZZ^\Pi$-grading of $\liea{u}$: the root spaces in
$\liea{n}_+$ have their usual weight in $\Phi^0_+ \subseteq \ZZ^{\Pi^0}$,
while the image of $\liea{g}_\beta,\ \beta \in \{0\} \cup \Phi^0_-,$ in
$\liea{g}/\liea{n}_+ \subseteq \liea{u}$ has weight $\delta+\beta$. Thus
the set of all weights ocurring in $\liea{u}$ is
\[ \Theta:=\Phi^0_+ \cup
    \{\delta + \beta \mid \beta \in \Phi^0_-\} \cup \{\delta\}. \]

\begin{thm} \label{thm:Affine}
Let $\Gamma$ be a simply laced Dynkin diagram of affine type from Figure
\ref{fig:Diagrams}, let $\Gamma^0$ be the subdiagram of finite type
obtained by removing vertex $0$, and let $\liea{g}$ be the Chevalley
algebra of type $\Gamma^0$ over a field of characteristic unequal
to $2$. For $x \in \Pi^0$ let $E_x^0 \in \liea{n}_+$ be the element
of the Chevalley basis of $\liea{g}$ with simple root $\alpha_x$
and for the lowest root $\theta \in \Phi^0_-$ let $E_{\theta}^0 \in
\liea{g}/\liea{n}_+$ be the image of the element in the Chevalley basis
of weight $\theta$. Then the map sending $x \in \Pi^0$ to $E_x^0$ and
$x_0$ to $E_{\theta}^0$ induces a $\ZZ^\Pi$-graded isomorphism $\mL(0)
\rightarrow \liea{n}_+ \ltimes \liea{g}/\liea{n}_+$ of Lie algebras.
\end{thm}

\begin{re}
Over $\CC$ one can argue directly in the Kac-Moody algebra $\gKM$. Then
$\mL(0)$ is also the quotient of the positive nilpotent subalgebra
of $\gKM$ by the root spaces with roots of height larger than the
Coxeter number $h$. In the proof one uses the root multiplicities
of \cite[Proposition 6.3]{Kac90}. One might also pursue this approach
in positive characteristic using the results of \cite{Billig91}, but
we have chosen to avoid defining the Kac-Moody algebra in arbitrary
characteristic and use the Chevalley basis instead.
\end{re}

\begin{proof}[Proof of Theorem \ref{thm:Affine}.]
The proof is close to that of Theorem \ref{thm:Finite}.
We start by verifying that the relations defining $\mL(0)$
hold in $\liea{u}=\liea{n}_+ \ltimes \liea{g}/\liea{n}_+$. First, $E_x^0$
and $E_y^0$ with $x,y \in \Pi^0$ commute when they are not connected in
$\Gamma$; this follows from the defining equations of $\gKM$. Second,
$E_x^0$ and $E_\theta^0$ commute if $x \in \Pi^0$ is not connected to $x_0$,
as $\theta+\alpha_x$ is then not in $\Phi^0$. Third, each $E_x^0$ is
a sandwich element in $\liea{u}$: for its action on $\liea{n}_+$ this
follows as in the proof of Theorem \ref{thm:Finite} and for its action
on $\liea{g}/\liea{n}_+$ it follows from the fact that $\ad(E_x^0)^2
\liea{g} \subseteq KE_x^0 \subseteq \liea{n}_+$. Fourth, $E_\theta^0$
is a sandwich element as $\ad(E_\theta^0)$ maps $\liea{u}$
into $\liea{g}/\liea{n}_+$, which has trivial multiplication. This
shows the existence of a homomorphism $\pi: \mL(0) \rightarrow
\liea{u}$. Moreover $\pi$ is graded; in particular, the weight of
$E_\theta^0$ is $\delta+\theta=\alpha_{x_0}$.

The $E^0_x$ generate $\liea{n}_+$ and $E_\theta^0$ generates the
$\liea{n}_+$-module $\liea{g}/\liea{n}_+$. These statements follow from
properties of the Chevalley basis in Subsection \ref{ssec:FiniteSandwich},
and imply that $\pi$ is surjective. So we need only show that $\dim \mL(0)
\leq \dim \liea{u}$; we prove this for each weight in $\Theta$.

First, the roots in $\Phi^0$ are very real, so their multiplicities
in $\mL(0)$ are at most $1$ by Lemma \ref{lm:VeryReal}. Second,
we claim that all roots of the form $\delta + \beta$ with $\beta
\in \Phi^0_-$ are also very real. This follows by induction
on the height of $\beta$: For $\beta=\theta$ it is clear since
$\delta+\theta=\alpha_{x_0}$. For $\beta \neq \theta$ it is well-known
that there exists an $x \in \Pi^0$ such that $\la \alpha_x,\beta
\ra=1$. Then we have $\delta+\beta=(\delta+\beta-\alpha_x)+\alpha_x$ where
$\delta+\beta-\alpha_x \in \Theta$ and $\la \alpha,\delta+\beta-\alpha_x
\ra=0 + 1 - 2=-1$; here we use that $\delta$ is in the radical of the
form $\la.,.\ra$. By induction, $\delta+\beta-\alpha_x$ is very real,
hence so is $\delta+\beta$. This shows that also the roots of the form
$\delta+\beta$ with $\beta \in \Phi^0_-$ have multiplicity at most $1$
in $\mL(0)$, again by Lemma \ref{lm:VeryReal}.

Next we show that $\delta$ has multiplicity at most $|\Pi^0|=\dim
\liea{h}$ in $\mL(0)$. Indeed, we claim that $\mL(0)_\delta$ is
contained in
\[ \sum_{x \in \Pi^0} [x,\mL(0)_{\delta-\alpha_x}]. \]
Then, by the above, each of the summands has dimension at most $1$,
and we are done. The claim is true almost by definition: any monomial of
weight $\delta$ must start with some $x \in \Pi$, so we need only show
that monomials starting with $x_0$ are already contained in the sum
above. Consider any monomial $z:=x_d \cdots x_1$ of weight $\delta$,
where $x_d=x_0$. As the coefficient of $x_0$ in $\delta$ is $1$, none
of the $x_i$ with $i<d$ is equal to $x_0$. But then an elementary
application of the Jacobi identity and induction shows that $z$ is a
linear combination of monomials that do {\em not} start with $x_0$.

Finally we have to show that $\mL(0)_\mu$ is $0$ for $\mu \in \Phi_+
\setminus \Theta$ (we already have $\mL(0)_\mu=0$ for $\mu \not \in
\Phi_+$ by Lemma \ref{lm:Roots}). But Lemma \ref{lm:Reduce} and the
fact that $\la \alpha_x, \delta \ra=0$ for all $x \in \Pi$ together
imply that $[x,\mL(0)_\delta]=0$ for all $x \in \Pi$. So it suffices to
show that if $\mu \in \Phi_+$ is not in $\Theta$, then ``$\mu$ can only
be reached through $\delta$''. More precisely: if $(x_d,\ldots,x_1)$
is any word over $\Pi$ such that $\sum_{j=1}^d \alpha_{x_j}=\mu$ and
$\mu_i:=\sum_{j=1}^i \alpha_{x_j} \in \Phi_+$ for all $i=1,\ldots,d$,
then there exists an $i$ such that $\mu_i=\delta$---but this follows
immediately from the fact that $\delta$ is the only root of height $h$
\cite[Proposition 6.3]{Kac90}. We find that every monomial corresponding
to such a word is zero, and this concludes the proof of the theorem.
\end{proof}

\section{The parameter space and generic Lie algebras}
\label{sec:Parameter}

So far we have only considered the Lie algebras $\mL(0)$. Now we
will be concerned with the variety $X$ of all parameters
$f \in (\mF^* )^\Pi$ for which $\dim \mL(f)=\dim \mL(0)$.
We collect some tools for determining $X$ in the case of simply laced
Dynkin diagrams.

\subsection{Scaling} \label{ssec:Scaling}
First let $\Gamma$ be arbitrary again, not necessarily a Dynkin
diagram. Scaling of the generators $x_i$ has an effect on $X$: 
Given $t=(t_x)_{x \in \Pi}$ in the torus $T:= (K^*)^\Pi$ there is a unique
automorphism of $\mF $ that sends $x \in \Pi$ to $t_x x$. This
gives an action of $T$ on $\mF $, and we endow $\mF^* $ with
the contragredient action. Finally, we obtain an action of $T$ on $X$ by
\[ (t f)_x (y):=t_x^{-1} f_x (t^{-1} y) \text{ for all }
    t \in T,\ f \in X,\ x \in \Pi,\text{ and } y \in
    \mF .
\]
Indeed, note that with this definition the automorphism of $\mF $
induced by $t$ sends $xxy-f_x(y)x \in \mF $ to
\begin{align*}
(tx)(tx)(ty) - f_x(y) tx&=t_x^2(xx(ty)-t_x^{-1}f_x(y)x)\\
&=t_x^2(xx(ty)-t_x^{-1}f_x(t^{-1}(ty))x)\\
&=t_x^2(xx(ty)-(tf)_x(ty)x),
\end{align*}
and hence the ideal $\mI(f) $ defining $\mL(f)$ to
$\mI(tf)$. Therefore, this automorphism of $\mF $ induces an
isomorphism $\mL(f) \rightarrow L(tf,\Gamma)$.

This scaling action of $T$ on $X$ will make things very easy in the
case of simply laced Dynkin diagrams, where $X$ will turn out to be
isomorphic to an affine space with linear action of $T$, in which the
maximal-dimensional orbits have codimension $0$, $1$ or $2$.

\begin{re}
Observe that the one-parameter subgroup $t \mapsto
\lambda(t):=(t,\ldots,t) \in T$ satisfies $\lim_{t \rightarrow \infty}
\lambda(t)f=0$ for all $f \in X$. This shows that all irreducible
components of $X$ contain $0$; in particular, $X$ is connected.
\end{re}

\subsection{The extremal form} \label{ssec:ExtremalForm}
In \cite{Cohen01} it is proved that on any Lie algebra over a field
of characteristic unequal to $2$ that is generated by finitely many
extremal elements, there is a unique bilinear form $\kappa$ such that
$xxy=\kappa(x,y)x$ for all extremal $x$. Moreover, it is shown there
that $\kappa$ is symmetric and associative: $\kappa(x,y)=\kappa(y,x)$
and $\kappa(xy,z)=\kappa(x,yz)$ for all $x,y,z$.  We call $\kappa$
the {\em extremal form}. For the Chevalley algebra $\liea{g}$
of Subsection \ref{ssec:FiniteSandwich} the extremal form is
non-zero on $\liea{g}_\alpha \times \liea{g}_\beta$ if and only if
$\alpha=-\beta$. The form may have a radical contained in the Cartan
subalgebra $\liea{h}$.

On the other hand, for any $f \in (\mF^* )^\Pi$ (not necessarily in $X$)
the Lie algebra $\mL(f)$ is generated by the images of the elements
of $\Pi$, which are extremal elements. In particular, for $x \in \Pi$
we have in $\mL(f)$
\[ f_x(y) x = xxy = \kappa(x,y) x, \]
where $\kappa$ is the extremal form on $\mL(f)$. So if the image of $x$
in $\mL(f)$ is non-zero, then $f_x(y)=\kappa(x,y)$.

\subsection{The Premet relations} \label{ssec:Premet}

Our arguments showing that certain monomials $m:=x_d \cdots x_1$ are
zero in the sandwich algebra $\mL(0)$ always depended on the sandwich
properties: $xxy=0$ and $xyxz=0$ whenever $x$ is a sandwich element and
$y,z$ are arbitrary elements of the Lie algebra. The Premet relations
of the following lemma translate such a statement into the following
statement: in $\mL(f)$ the monomial $m$ can be expressed in terms
of monomials of degree less than $d-1$ and values of $f_{x_d}$ on
monomials of degree less than $d-1$.

\begin{lm}[\cite{Chernousov89}]
Let $x$ be a non-zero extremal element of a Lie algebra $\mL$, and let $f_x:\mL
\rightarrow K$ be the linear function with $xxy=f_x(y)x$. Then we have
\[ 2xyxz=f_x(yz)x - f_x(z)xy - f_x(y)xz. \]
\end{lm}

\begin{re}
In characteristic $2$ the definition of an extremal element $x$
involves the existence of a function $g_x$ such that $xyxz=g_x(yz)x
- g_x(z)xy - g_x(y)xz$, i.e., $g_x$ plays the role of $f_x/2$. See
\cite[Definition 14]{Cohen06}.
\end{re}

\subsection{The parameters}
Recall from Section \ref{sec:Moduli} that the restriction map $X
\rightarrow (V^*)^\Pi$ is injective and has a closed image; a key step
in the proof was showing that for $f \in X$ the values $f_x(u),x \in
\Pi,u \in F$ depend polynomially on $f|_V$. In what follows,
this will be phrased informally as {\em $f$ can be expressed in $f|_V$}
or {\em $f|_V$ determines $f$}. In this phrase we implicitly make the
assumption that $f \in X$, i.e., that $\mL(f)$ has the maximal possible
dimension. In the case of Dynkin diagrams, we will exhibit a small number
of values of $f$ in which $f$ can be expressed. For this purpose, the
following lemma, which also holds for other graphs, is useful.

\begin{lm} \label{lm:Expressible}
Let $q=x_d \cdots x_1$ be a monomial of degree $d \geq 2$ and weight
$\beta$, and let $z \in \Pi$ be such that $\la \alpha_z, \beta \ra
\geq -1$. Then $f_z(q)$ can be expressed in the parameters $f_x(m)$
with monomials $m$ of degree less than $d-1$ and $x \in \Pi$.
\end{lm}

\begin{proof}
First, if $x_d$ is not a
$\Gamma$-neighbour
of $z$ in $\Gamma$, then
\[ f_z(x_d\cdots x_1)=\kappa(z,x_d \cdots x_1)=-\kappa(x_d z,x_{d-1}
\cdots x_1)=\kappa(0,x_{d-1}\cdots x_1)=0, \]
and we are done. So assume that $x_d$ is a
$\Gamma$-neighbour
of $z$. Now
\begin{align*}
f_z(x_d\cdots x_1)&=-\kappa(x_d z, x_{d-1} \cdots x_1)=\kappa(z x_d, x_{d-1}\cdots x_1)\\
&=-f(x_d,z x_{d-1} \cdots x_1)=f_{x_d}(z x_{d-1}\cdots
x_1).
\end{align*}
In both cases we have used that the images of $z$ and $x_d$ are non-zero
in $\mL(f)$ for $f \in X$; see Subsection \ref{ssec:ExtremalForm}.
Now $\la \alpha_z, \beta-\alpha_{x_d} \ra \geq 0$, so  Lemma
\ref{lm:Reduce} says that $z x_{d-1} \cdots x_1$ can be expressed in terms
of smaller monomials and values $f_x(m)$ for $x \in \Pi$ and monomials
$m$ of degree less than $d-1$. Then by linearity of $f_{x_d}$ the last
expression above can also be expressed in terms of values $f_x(m)$
with $x \in \Pi$ and $m$ of degree less than $d-1$.
\end{proof}

\subsection{Simply laced Dynkin diagrams of finite type}

Suppose that $\Gamma$ is a simply laced Dynkin diagram of finite type. Let
$\liea{g}$ be the Chevalley algebra with Dynkin diagram $\Gamma$ of
Subsection \ref{ssec:FiniteSandwich}. We identify $\ZZ^\Pi$ with the
character group of $T=(K^*)^\Pi$ in the natural way: we write $\mu$ for
the character that sends $t$ to $t^\mu=\prod_{x \in \Pi} t_x ^{\mu_x}$.
Furthermore, let $\Sigma$ be the set of edges of $\Gamma$, and write
$\alpha_e:=\alpha_x+\alpha_y$ for $e=\{x,y\} \in \Sigma$.

\begin{thm} \label{thm:GenericFinite}
Let $\Gamma=(\Pi,\Sigma)$ be a simply laced Dynkin diagram of finite type,
obtained from a diagram in Figure \ref{fig:Diagrams} by removing vertex
$0$. Let $\liea{g}$ be the Chevalley algebra of type $\Gamma$ over the
field $K$ of characteristic unequal to $2$, and set $T:=(K^*)^\Pi$. Then
the variety $X$ is, as a $T$-variety, isomorphic to the vector space
$V:=K^\Sigma$ on which $T$ acts diagonally with character $-\alpha_e$
on the component corresponding to $e \in \Sigma$.  For $f$ corresponding
to any element in the dense $T$-orbit $(K^*)^\Sigma$ the Lie algebra $\mL(f)$
is isomorphic to a fixed Lie algebra.
\end{thm}

We first need a lemma that will turn out to describe the generic
$\mL(f)$.  We retain the notation $E_x^0,H_x^0,F_x^0 \in \liea{g}$
and $\liea{n}_+$ from Subsection \ref{ssec:FiniteSandwich}.
We denote by $C$ the variety of tuples $(G_x)_{x \in \Pi}$ with $G_x
\in \la F_x^0, H_x^0, E_x^0 \ra \cong \liea{sl}_2$ extremal; $C$ is an
irreducible variety.

\begin{lm} \label{lm:GenericFinite}
For generic $G=(G_x)_{x \in \Pi} \in C$ the Lie subalgebra $\liea{g}'$
of $\liea{g}$ generated by the $G_x$ has dimension $\dim \liea{n}_+$, and
moreover $G_x G_x G_y$ is a non-zero multiple of $G_x$ for all $x \sim y$.
\end{lm}

\begin{proof}
By definition $\liea{g}'$ is generated by extremal elements, hence it has
dimension at most that of $\mL(0)$, which is isomorphic to $\liea{n}_+$
by Theorem \ref{thm:Finite}. The condition that the $G_x$ generate a
Lie algebra of dimension less than $\dim \liea{n}_+$ is closed, and
the tuple $(E_x^0)_{x \in \Pi} \in C$ does not fulfill it.  Hence using
the irreducibility of $C$ we find that for $G$ in an open 
dense subset of $C$ the Lie algebra $\liea{g}'$ has
$\dim \liea{g}'=\dim \liea{n}_+$.  This proves the first statement. The
second statement follows directly from the same statement for $\Gamma$
of type $A_2$, i.e., for $\liea{g}=\liea{sl}_3$, where it boils down
to the statement that the two copies of $\liea{sl}_2$ in $\liea{sl}_3$
corresponding to the simple roots are not mutually perpendicular relative
to the extremal form in $\liea{sl}_3$.
\end{proof}

\begin{proof}[Proof of Theorem \ref{thm:GenericFinite}]
By Lemma \ref{lm:Expressible} and Theorem \ref{thm:Finite}, any $f \in X$
is determined by its values $f_x(m)$ with $x \in \Pi$ and monomials $m$
of weights $\beta \in \Phi_+$ such that either $\beta$ has height $1$ or
$\la \alpha_x, \beta \ra \leq -2$. But since $\beta$ is a positive root,
the latter inequality cannot hold. Hence $\beta$ has height $1$, so that
$m \in \Pi$, and the only $x \in \Pi$ for which $f_x(m) \neq 0$ are the
neighbours of $m$. Moreover, from the symmetry of the extremal form we
conclude that $f_x(y)=f_y(x)$ for $x,m \in \Pi$ neighbours in $\Gamma$.

We have thus found a closed embedding $\Psi: X \rightarrow K^\Sigma$ sending
$f$ to $(f_x(y))_{\{x,y\} \in \Sigma}$. Now if we let $T$ act on $K^\Sigma$
through the homomorphism
\[ T \rightarrow (K^*)^\Sigma,\ t \mapsto (t_x^{-1} t_y^{-1})_{\{x,y\} \in \Sigma}, \]
then $\Psi$ is $T$-equivariant by the results of Subsection
\ref{ssec:Scaling}. Note that $T$ acts by the character
$-\alpha_e$ on the component corresponding to $e \in \Sigma$.
The fact that $\Gamma$ is a tree readily implies that the characters
$\alpha_e,\ e \in \Sigma,$ are linearly independent over $\ZZ$ in the character
group of $T$, so that the homomorphism $T \rightarrow (K^*)^\Sigma$ is
surjective. But then $T$ has finitely many orbits on $K^\Sigma$, namely of
the form $(K^*)^{\Sigma'} \times \{0\}^{\Sigma \setminus \Sigma'}$ with $\Sigma'\subseteq \Sigma$.

Now as $\Psi(X)$ is a closed $T$-stable subset of $K^\Sigma$ we are done
if we can show that $(K^*)^\Sigma \cap \Psi(X)$ is non-empty.  But this is
precisely what Lemma \ref{lm:GenericFinite} tells us: there exist Lie
algebras $\liea{g}'$ generated by extremal elements that have the largest
possible dimension, and where all coordinates $f_x(y)$ with $x \sim y$
are non-zero. This concludes the proof.
\end{proof}

\begin{re}
The proof above also implies that all Lie algebras described in Lemma
\ref{lm:GenericFinite} are isomorphic. More generally: for any two Lie
algebras $\liea{g}'$ and $\liea{g}''$ with tuples of distinguished,
extremal generators $(G'_x)_{x \in \Pi}$ and $(G''_x)_{x \in \Pi}$
such that
\begin{enumerate}
\item $G'_x G'_y=0$ and $G''_x G_y''=0$ for $x \not \sim y$,
and
\item $G'_x G'_x G'_y = 0$ if and only if $G''_x G''_x G''_y = 0$
for $x \sim y$, and
\item $\dim \liea{g}'=\dim \liea{g}''=\dim \liea{n}_+$,
\end{enumerate}
there exists an isomorphism $\liea{g}' \rightarrow \liea{g}''$ mapping each
$G'_x$ to a scalar multiple of $G''_x$.
\end{re}

\subsection{Simply laced Dynkin diagrams of affine type}
Suppose now that $\Gamma$ is a simply laced Dynkin diagram of affine
type. We retain the notation from Subsection \ref{ssec:AffineSandwich}. In
particular, let $\liea{g}$ be the Chevalley algebra of type $\Gamma^0$,
the graph induced by $\Gamma$ on $\Pi^0=\Pi \setminus \{x_0\}$.
To state the analogue of Theorem \ref{thm:GenericFinite} we again
identify $\ZZ^\Pi$ with the character group of $T=(K^*)^\Pi$ and retain
the notation $\alpha_e$ for $e \in \Sigma$, the edge set of $\Gamma$.

\begin{thm} \label{thm:GenericAffine}
Let $\Gamma=(\Pi,\Sigma)$ be a simply laced Dynkin diagram of affine type
from Figure \ref{fig:Diagrams} and let $\Gamma^0$ be the finite-type
diagram obtained by deleting vertex $0$ from $\Gamma$. Let $\liea{g}$
be the Chevalley algebra of type $\Gamma^0$ over the field $K$ of
characteristic unequal to $2$, and set $T:=(K^*)^\Pi$. Then the variety
$X$ is, as a $T$-variety, isomorphic to the vector space $V:=K^\Sigma
\times K$ on which $T$ acts diagonally with character $-\alpha_e$ on the
component corresponding to $e \in \Sigma$, and with character $-\delta$
on the last component. For all $f \in X$ corresponding to points in
some open dense subset of $K^\Sigma \times K$ the Lie algebra $\mL(f)$
is isomorphic to $\liea{g}$.
\end{thm}

\begin{re} \label{re:Rank}
Unlike for diagrams of finite type, it is {\em not} necessarily true
that $T$ has only finitely many orbits on $V$. Indeed, the following
three situations occur:

\begin{enumerate}
\item The characters $\alpha_e\ (e \in \Sigma), \delta$ are linearly
independent. This is the case for $D_{\text{even}}^{(1)}$, $E_7^{(1)}$,
and $E_8^{(1)}$. Then $T$ has finitely many orbits on $V$.

\item The characters $\alpha_e\ (e \in \Sigma)$ are linearly independent,
but $\delta$ is in their $\QQ$-linear span. This is the case for
$A_{\text{even}}^{(1)}$, $D_{\text{odd}}^{(1)}$ and $E_6^{(1)}$. Now
the orbits of $T$ in $(K^*)^\Sigma \times K^*$ have codimension $1$. For
\text{$A_{even}^{(1)}$ and $E_6^{(1)}$} the character $\delta$ has
full support when expressed in the $\alpha_e$; and this readily implies
that $T$ has finitely many orbits on the complement of $(K^*)^\Sigma \times
K^*$. For $D_{n}^{(1)}$ with $n$ odd, however, $\frac{n-3}{2}$ edge
characters get coefficient $0$ when $\delta$ is expressed in them,
and therefore $T$ still has infinitely many orbits on said complement.

\item The characters $\alpha_e\ (e \in \Sigma)$ are linearly dependent. This is
the case only for $A_{\text{odd}}^{(1)}$, and in fact $\delta$ is then also
in the span of the $\alpha_e$. Now the $T$-orbits in $(K^*)^\Sigma \times K^*$
have codimension $2$, and in the complement there are still infinitely
many orbits.
\end{enumerate}

This gives some feeling for the parameter space $X$. It would be
interesting to determine exactly all isomorphism types of Lie algebras
$\mL(f)$ with $f \in X$---but here we confine ourselves to those with $f$
in some open dense subset of $K^\Sigma \times K$.
\end{re}

The proof is very similar to that of Theorem \ref{thm:GenericFinite}.
Again, we first prove a lemma that turns out to describe the generic
$\mL(f)$. Retain the notation $E^0_x,H^0_x,F^0_x \in \liea{g}$
for $x \in \Pi^0$. Moreover, denote the lowest weight by $\theta \in
\Phi^{0}_-$, let $E^0_{x_0},F^0_{x_0} \in \liea{g}$ be the elements of
the Chevalley basis of weights $\theta$ and $-\theta$, respectively, and
set $H^0_{x_0}:=[E^0_{x_0},F^0_{x_0}]$. Write $C$ for the irreducible
variety of tuples $(G_x)_{x \in \Pi}$ with $G_x \in \la F^0_x, H^0_x,
E^0_x \ra \cong \liea{sl}_2$ extremal.

\begin{lm} \label{lm:GenericAffine}
For generic $G=(G_x)_{x \in \Pi}\in C$ the $G_x$ generate $\liea{g}$, and
moreover $G_x G_x G_y$ is a non-zero multiple of $G_x$ for all $x \sim y$.
\end{lm}

\begin{proof}
The first statement is true for $G=(E^0_x)_{x\in \Pi}$; this
follows from the properties of the Chevalley basis in Subsection
\ref{ssec:FiniteSandwich}. Hence by the irreducibility of $C$
it is true generically. The second statement follows, as in Lemma
\ref{lm:GenericFinite}, from the same statement in $\liea{sl}_3$.
\end{proof}

In the following proof we will show that the choice of $(G_x)_{x \in\Pi}$
as in Lemma \ref{lm:GenericAffine} already gives generic points in $X$,
except for the case where $\Gamma$ is of type $A_{odd}^{(1)}$, for which
we give another construction.

\begin{proof}[Proof of Theorem \ref{thm:GenericAffine}]
By Lemma \ref{lm:Expressible} and Theorem \ref{thm:Affine}, any $f \in X$
is determined by its values $f_x(m)$ with $x \in \Pi$ and monomials $m$
of weights $\beta \in \Theta$ such that either $\beta$ has height $1$
or $\la \alpha_x, \beta \ra \leq -2$. In contrast with the case of
finite-type diagrams, there do exist pairs $(x,\beta) \in \Pi \times
\Theta$ with this latter property, namely precisely those of the form
$(x,\delta-\alpha_x)$. For all $x \in \Pi$, let $m_x$ be a monomial that
spans the weight space in $\mL(0)$ of weight $\delta-\alpha_x$; this
space is $1$-dimensional by Theorem \ref{thm:Affine}. We claim that the
$f_x(m_x)$ can all be expressed in terms of $f_{x_0}(m_{x_0})$ and values
$f_z(r)$ with $z \in \Pi$ and $r$ of degree less than $h-2$. Indeed, if $x
\neq x_0$, then $x_0$ occurs exactly once in $m_x$; and writing $m_x=x_d
\cdots x_1 x_0 y_e \cdots y_1$ with $x_1,\ldots,x_d,y_1,\ldots,y_e \in
\Pi^0$ we find
\begin{align*}
f_x(m_x)&=\kappa(x,x_d\cdots x_1 x_0 y_e \cdots y_1)\\
    &=(-1)^{d+1}\kappa(x_0 x_1 \cdots x_d x,y_e \cdots y_1)\\
    &=(-1)^d \kappa(x_0,[x_1 \cdots x_d x,y_e \cdots y_1])\\
    &=(-1)^d f_{x_0}([x_1 \cdots x_d x,y_e \cdots y_1]),
\end{align*}
and the expression $[x_1 \cdots x_d x,y_e \cdots y_1]$ can be rewritten in
terms of $m_{x_0}$ and shorter monomials, using values $f_z(r)$ with $r$
of degree less than $d+e=h-2$.

We have now found a closed embedding $X \rightarrow K^\Sigma
\times K$ which sends $f$ to $\left((f_x(y))_{\{x,y\} \in
\Sigma},f_{x_0}(m_{x_0})\right)$; for ease of exposition we will view
$X$ as a closed subset of $K^\Sigma \times K$.  The theorem follows
once we can realise generic parameter values in $K^\Sigma \times K$
with extremal elements that generate $\liea{g}$.  To this end, choose a
generic tuple $(G_x)_{x \in\Pi}$ in $C$. By Lemma \ref{lm:GenericAffine}
these generate $\liea{g}$ and they clearly satisfy $G_x G_y=0$ for $x \not
\sim y$. Hence they yield a point in $X$ with $f_x(y)=\kappa(G_x,G_y)\neq
0$ for $x \sim y$. Furthermore, the parameter $f_{x_0}(m_{x_0})$ equals
the extremal form evaluated on $G_{x_0}$ and the monomial $m_{x_0}$
evaluated in the $G_x$. Express that monomial in the $G_x$ as $\xi
E^0_{x_0} + \eta H^0_{x_0} + \zeta F^0_{x_0}$ plus a term perpendicular
to $\la E^0_{x_0},H^0_{x_0},F^0_{x_0} \ra$, and write $G_{x_0}$ as $\xi'
E^0_{x_0} + \eta' H^0_{x_0}+ \zeta' F^0_{x_0}$. For the degenerate case
where $G_x=E^0_x$ for all $x$ we have $\xi,\eta=0$ and $\zeta \neq 0$
(that monomial is a non-zero scalar multiple of the highest root vector,
$F^0_{x_0}$), so that $f_{x_0}(m_{x_0})=\kappa(E^0_{x_0},\zeta F^0_{x_0})
\neq 0$.  Therefore, this parameter is non-zero generically. Hence we
have found a point $f \in X \cap ((K^*)^\Sigma \times K^*)$. Along the
lines of Remark \ref{re:Rank} we now distinguish three cases:

\begin{enumerate}
\item
If the $\alpha_e (e \in \Sigma)$ and $\delta$ are linearly independent,
then $T$ acts transitively on $(K^*)^\Sigma \times K^*$, and we are done.

\item If the $\alpha_e (e \in \Sigma)$ are linearly independent, but $\delta$
lies in their span, then we now show that we can alter the point $f$
above in a direction transversal to its $T$-orbit.  Let $S$
be the torus in the adjoint group of $\liea{g}$ whose Lie algebra is
$\liea{h}$, and consider the effect on $f$ of conjugation of $G_{x_0}$
with an element $s \in S$, while keeping the other $G_x$ fixed. This
transforms $G_{x_0}=\xi' E^0_{x_0} + \eta' H^0_{x_0} + \zeta' F^0_{x_0}$
in $s^{\theta} \xi' E^0_{x_0} + \eta' H^0_{x_0} + s^{-\theta} \zeta'F_{x_0}$, and therefore it transforms $f_{x_0}(m_{x_0})$ into
\[ s^{\theta} \xi' \zeta \kappa(E^0_{x_0},F^0_{x_0}) +
\eta' \eta \kappa(H^0_{x_0},H^0_{x_0}) +
s^{-\theta} \zeta' \xi \kappa(F^0_{x_0},E^0_{x_0}),\]
while it keeps the parameters $f_{x_0}(y)$ with $x_0 \sim y$ unchanged:
these only depend on $\eta'$. This shows that we can indeed move $f$
inside $X$ in a direction transversal to its $T$-orbit, and we are done.

\label{it:Conjugate}

\item Finally, in the case of $A_{n-1}^{(1)}$ with $n$ even we first
show that tuples in $C$ only give points in a proper closed subset of
$K^\Sigma \times K$. Here $\liea{g}=\liea{sl}_n$ and $\Gamma$ is an
$n$-cycle; label its points $0, \ldots, n-1$. Relative to the usual
choices of $E^0_i,H^0_i,F^0_i$  the element $G_i$ is a matrix with
$2 \times 2$-block
\[ \begin{bmatrix}
    a_i b_i & a_i^2 \\
    -b_i^2 & -a_ib_i
\end{bmatrix} \]
on the diagonal in rows (and columns) $i$ and $i+1$ and zeroes
elsewhere. We count the rows and columns modulo $n$ so that row
$0$ is actually row $n$. But then we have $\kappa(G_i,G_{i+1})=2
a_ib_ia_{i+1}b_{i+1}$, and this implies
\begin{align} \label{eq:SubVariety}
&\kappa(G_1,G_2)\kappa(G_3,G_4)\cdots\kappa(G_{n-1},G_0) \\
&=2^n(a_1 b_1)(a_2 b_2)(a_3b_3)(a_4 b_4)\cdots(a_{n-1}b_{n-1})(a_0b_0) \notag\\
&=2^n(a_0 b_0)(a_1 b_1)(a_2b_2)(a_3 b_3)\cdots(a_{n-2}b_{n-2})(a_{n-1}b_{n-1})\notag \\
&=\kappa(G_0,G_1)\kappa(G_2,G_3)\cdots\kappa(G_{n-2},G_{n-1});
\notag
\end{align}
so the tuple of parameter values of the tuple $(G_i)_{i=0}^{n-1} \in C$
lies in a proper closed subset $R$ of $K^\Sigma \times K$.

We therefore allow the tuple $(G_i)_{i=0}^{n-1}$ to vary in a
slightly larger variety $C' \supset C$ as follows: the conditions on
$G_1,\ldots,G_{n-1}$ remain the same, but $G_0$ is now allowed to take
the shape
\[
\begin{bmatrix}
-a_0 b_0 & 0 & \ldots & 0 & -b_0^2\\
c_2 a_0 & 0 & \ldots & 0 & c_2 b_0\\
\vdots  & \vdots & & \vdots & \vdots\\
c_{n-1} a_0 & 0 & \ldots & 0 & c_{n-1} b_0\\
a_0^2 & 0 & \ldots & 0 & a_0 b_0
\end{bmatrix}
\]
(which is extremal since it has rank $1$ and trace $0$), subject to
the equations
\begin{equation} \label{eq:abc}
b_i c_i + a_i c_{i+1} =0 \text{ for } i=2,\ldots,n-2,
\end{equation}
which ensure that $G_0$ commutes with $G_2,\ldots,G_{n-2}$. Still, any
tuple in an open neighbourhood $U \subseteq C'$ of our original tuple
$(G_i)_{i=0}^{n-1}$ (with generic $a_i$ and $b_i$ but all $c_i$ equal
to $0$) generates $\liea{sl}_n$. We now argue that the differential $d$
at $(G_i)_{i=0}^{n-1}$ of the map $U \rightarrow X \subseteq K^\Sigma
\times K$ sending a tuple to the parameters that it realises has rank
$|\Sigma|+1$, as required. Indeed, the $T$-action already gives a
subspace of dimension $|\Sigma|-1$, tangent to $R$.  Making $c_2$ (and
hence all $c_i$) non-zero adds $-2 a_1^2 c_2 a_0$ to $\kappa(G_0,G_1)$
and $2 b_{n-1}^2 c_{n-1}b_0$ to $\kappa(G_0,G_{n-1})$, and it fixes all
other $\kappa(G_i,G_j)$. We show that this infinitesimal direction is
not tangent to $R$: it adds
\[ 2^n b_{n-1}^2 c_{n-1}b_0 (a_1 b_1)(a_2 b_2)\cdots
	(a_{n-3}b_{n-3})(a_{n-2}b_{n-2})
\]
to the left-hand side of \eqref{eq:SubVariety}, and
\[ -2^n a_1^2 c_2a_0 (a_2b_2)(a_3 b_3)\cdots(a_{n-2}b_{n-2})(a_{n-1}b_{n-1})\]
to the right-hand side. Dividing these expressions by common
factors, the first becomes $2^n b_{n-1}c_{n-1}b_0b_1$ and the second
$-2^n a_1 c_2 a_0 a_{n-1}$. These expressions are not equal
generically, even modulo the equations \eqref{eq:abc} relating the
$c_i$ to the $a_i$ and $b_i$; indeed, these equations do not involve
$a_0,a_1,a_{n-1},b_0,b_1,b_{n-1}$.

Note that varying $c_2$ may also effect the parameter $f_{x_0}(m_{x_0})$,
but in any case the above shows that the composition of the differential
$d$ with projection onto $K^\Sigma$ is surjective. On the other hand,
conjugation with the torus $S$ as in case \eqref{it:Conjugate} yields
a vector in the image of $d$ which is supported only on the factor $K$
corresponding to $\delta$. This concludes the proof that $d$ has full
rank.
\end{enumerate}
\end{proof}

\section{Notes}
\label{sec:Notes}

\subsection{Recognising the simple Lie algebras}
Going through the proof that $X$ is an affine variety,
one observes that the map $f \mapsto f|_V$ is not only injective on
$X$, but even on
\[ X'(\Gamma):=\{ f \in (\mF^*)^\Pi \mid \forall x \in \Pi: x \neq 0 \text{
in } \mL(f) \} \supseteq X. \]
The same is true for the map $f \mapsto \left((f_x(y))_{\{x,y\}
\in \Sigma}\right)$ in the case where $\Gamma$ is a Dynkin diagram of
finite type, and for the map $f \mapsto \left((f_x(y))_{\{x,y\} \in
\Sigma},f_{x_0}(m_{x_0})\right)$ in the case where $\Gamma$ is a Dynkin
diagram of affine type. This shows that, for these Dynkin diagrams,
$X'(\Gamma)$ is actually {\em equal} to $X$, whence the
following theorem.

\begin{thm}
Suppose that $\Gamma$ is a Dynkin diagram of finite or affine type.
Let $\mL$ be any Lie algebra, over a field of characteristic unequal
to $2$, which is generated by non-zero extremal elements $G_x, x \in
\Pi$, in which the commutation relations $G_x G_y=0$ for $x \not \sim y$
hold. Define $f \in (\mF^*)^\Pi$ by the condition that $G_x G_x u = f_x(u)
G_x$ holds in $\mL$. Then $f \in X$ and $\mL$ is a quotient of $\mL(f)$.
\end{thm}

This theorem could well prove useful for recognising the
Chevalley algebras $\liea{g}$: if $f$ corresponds to a point in
the open dense subset of $K^\Sigma \times K$ referred to in Theorem
\ref{thm:GenericAffine}, then one concludes that $\mL$ is a quotient
of $\liea{g}$. Hence if $\liea{g}$ is a simple Lie algebra, then $\mL$
is isomorphic to $\liea{g}$.  It is not clear to us whether, for general
$\Gamma$, the image of $X'(\Gamma)$ in $(V^*)^\Pi$ is closed; this is
why we chose to work with $X$ instead.

\subsection{Other graphs}
Our methods work very well for Dynkin diagrams, but for more general
graphs new ideas are needed to determine $\mL(0),X,$ and
$\mL(f)$ for $f \in X$. The relation with the Kac-Moody
algebra of $\Gamma$ may be much tighter than we proved in Subsection
\ref{ssec:KacMoody}. General questions of interest are: Is $X$
always an affine space? Is there always a generic Lie algebra? We
expect the answers to both questions to be negative, but do not have
any counterexamples.

The references \cite{panhuis07,Postma07,Roozemond05} contain other series
of graphs which exhibit the same properties as we have proved here:
the variety $X$ is an affine space, and generic points in it
correspond to simple Lie algebras of types $A_n,C_n,B_n,D_n$.  In fact,
the graph that they find for $C_n$ is just the finite-type Dynkin
diagram of type $A_{2n}$. This also follows easily from our results:
take $2n$ generic extremal elements $(G_x)_x$ in $\liea{sl}_{2n+1}$
as in Lemma \ref{lm:GenericFinite}. These generate a subalgebra of
$\liea{sl}_{2n+1}$ of dimension $\binom{2n+1}{2}$ by that same lemma,
and if we consider them as matrices, their images span a subspace $W$ of
dimension $2n$ in $K^{2n+1}$.  It is not hard to write down an explicit,
non-degenerate skew symmetric form on $W$ with respect to which the $G_x$
are skew---hence the Lie algebra generated by them is $\liea{sp}_{2n}$.

%\bibliographystyle{amsplain}
%\bibliography{diffeq}

\providecommand{\bysame}{\leavevmode\hbox to3em{\hrulefill}\thinspace}
\providecommand{\MR}{\relax\ifhmode\unskip\space\fi MR }
\providecommand{\MRhref}[2]{%
  \href{http://www.ams.org/mathscinet-getitem?mr=#1}{#2}
}
\providecommand{\href}[2]{#2}

\end{document}